\documentclass{rmmcart}
\usepackage{dsfont}
\usepackage{amssymb}
\usepackage{verbatim}
\input xy
\xyoption{all}
\everymath{\displaystyle}
\newtheorem{theorem}{Theorem}[section]
\newtheorem{lemma}[theorem]{Lemma}
\newtheorem{proposition}[theorem]{Proposition}

\theoremstyle{definition}
\newtheorem{definition}[theorem]{Definition}
\newtheorem{eg}[theorem]{Example}

\newtheorem{remark}[theorem]{Remark}
\newtheorem{corollary}[theorem]{Corollary}

\numberwithin{equation}{section}

\newcommand{\IC}{\mathbb{C}}

 \newcommand{\IN}{\mathbb{N}}

 \newcommand{\IT}{\mathbb{T}}

\newcommand{\CA}{\mathcal{A}}
\newcommand{\CB}{\mathcal{B}}
\newcommand{\CC}{\mathcal{C}}

\newcommand{\CH}{\mathcal{H}}

\newcommand{\CL}{\mathcal{L}}
\newcommand{\CM}{\mathcal{M}}

\newcommand{\CS}{\mathcal{S}}

\newcommand{\CK}{\mathcal{K}}

\newcommand{\CV}{\mathcal{V}}
\newcommand{\CX}{\mathcal{X}}

\newcommand{\ls}{(E,\mathcal{L}, \mathcal{B})}

\title{On Labeled Graph $C^*$-algebras}

\author{Debendra P Banjade}  
\address{Department of Mathematics and Statistics, Coastal Carolina
University, Conway, SC 29528--6054}
\email{dpbanjade@coastal.edu}   

\author{Menassie Ephrem}  
\address{Department of Mathematics and Statistics, Coastal Carolina
University, Conway, SC 29528--6054}
\email{menassie@coastal.edu}   

\thanks{The authors would like to thank the referee for the thorough review and suggestions. Also, the authors would like to thank Andrew Incognito for important feedback on the first version of this manuscript.}

\thanks{The second author would like to thank Mark Tomforde for the introduction of the topic and for his continued support.}



\begin{document}

\date{July, 2019.}

\begin{abstract}

Given a directed graph $E$ and a labeling $\mathcal{L}$, one forms the labeled graph $C^*$-algebra by taking a weakly left--resolving labeled space  $(E, \mathcal{L}, \mathcal{B})$ and considering a universal generating family of partial isometries and projections.  In this paper, we work on ideals for a labeled graph $C^*$-algebra when the graph contains sinks.  Using some of the tools we build, we compute $C^*(E, \mathcal{L}, \mathcal{B})$ when $E$ is a finite graph.

\end{abstract}

\maketitle

\section{Introduction} \label{sect.z21}

For many decades graphs have been used as a tool to study a large class of $C^*$-algebras.  In \cite{KPRR}, Kumjian, Pask, Raeburn and Renault defined the graph groupoid of a countable row--finite directed graph with no sinks, and showed that the $C^*$-algebra of this groupoid coincided with a universal $C^*$-algebra generated by partial isometries satisfying relations naturally generalizing those given in \cite{Cu}.  
More recently several people have worked on generalizing these results to arbitrary directed graphs, higher-rank graphs, and ultra--graphs.  

In the early 2000s, Tomforde introduced ultra--graph $C^*$-algebras.  Bates and Pask, in \cite{BP1} introduced a new class of $C^*$-algebras called $C^*$-algebras of labeled graphs.  Later, in a series of papers (along with Carlsen) they provided some classification of these algebras including computation of $K$-theory.  More recently,  Jeong, Kang and Kim, in \cite{JKK}, among other interesting results, they provided a characterization for labeled graph $C^*$-algebra to be an AF algebra.

A directed graph $E = (E^0, E^1,~~ s,~~ r)$ consists of a
countable set $E^0$ of vertices and $E^1$ of edges, and maps $s,r:
E^1 \rightarrow E^0$ identifying the source (origin) and the range
 (terminus) of each edge. The graph is row--finite if each
vertex emits at most finitely many edges. A vertex is a sink if it
is not a source of any edge.  A path is a
sequence of edges $e_1e_2\ldots e_n$ with $r(e_i) = s(e_{i+1})$
for each $i = 1,2,\ldots , n-1$. An infinite path is a sequence
$e_1e_2\ldots$ of edges with $r(e_i)=s(e_{i+1})$ for each $i$.

For a finite path $p=e_1e_2 \ldots e_n$, we define $s(p): =s(e_1)$
and $r(p):= r(e_n)$. For an infinite path $p=e_1e_2 \ldots$, we
define $s(p):=s(e_1)$. We use the following notations:

$E^*:=\bigcup_{n=0}^\infty E^n,$ where $E^n$ := $\{p:p \text{ is a path of length } n\}.$

$E^{**}~~ :=~~ E^* \cup E^\infty$, where $E^\infty$ is the set of infinite paths.

Several works have been done on labeled graph $C^*$-algebras with the restriction that the graph has no sinks. In this paper, we will present results on these algebras when the graph may have sinks.  The paper is organized as follows.  In section 2, we develop some terminology for labeled graphs.  Then, in section 3, we briefly describe labeled graph $C^*$-algebras.  In section 4, after building the tools needed, we provide the theorems that characterize ideals of labeled graph $C^*$-algebras and finally describe the $C^*$-algebras for finite graphs with no loops or cycles.

\section{Preliminaries}\label{preliminaries}

Let $E = (E^0, ~E^1, ~s, ~r)$ be a directed graph and let $\CA$ be an alphabet.  A labeling is a function $\CL : E^1 \longrightarrow \CA$.  Without loss of generality, we will assume that $\CA = \CL(E^1)$.  The pair $(E,~\CL)$ is called a labeled graph.

Given a labeled graph $(E, ~\CL$), we extend the labeling function $\CL$ canonically to the sets $E^*$ and $E^\infty$ as follows.
Using $\CA^n$ for the set of words of size $n$, $\CL$ is defined from $E^n$ into $\CA^n$ as $\CL(e_1e_2\ldots e_n) = \CL(e_1)\CL(e_2)\ldots \CL(e_n)$.  Similarly, for $p = e_1e_2\ldots \in E^\infty$, $\CL(p) = \CL(e_1)\CL(e_2)\ldots \in \CA^\infty$.

Following tradition, we use $\CL^*(E):=\bigcup_{n=1}^\infty \CL(E^n)$, and $\CL^\infty(E) := \CL(E^\infty)$.

For a word $\alpha = a_1a_2\ldots a_n \in \CL^n(E)$, we write $$s(\alpha) := \{s(p):p\in E^n,~ \CL(p) = \alpha\}$$ and $$r(\alpha) := \{r(p):p\in E^n,~ \CL(p) = \alpha\}.$$  Similarly for $\alpha = a_1a_2\ldots \in \CL^\infty(E)$, $$s(\alpha) := \{s(p):p\in E^\infty,~ \CL(p) = \alpha\}.$$

Each of these sets is a subset of $E^0$.  The use of $s$ and $r$ for an edge/path verses a label/word should be clear from the context.

A labeled graph $(E,~\CL)$ is said to be left--resolving if for each $v\in E^0$, the function $\CL:r^{-1}(v) \rightarrow \CA$ is injective.  In other words, no two edges pointing to the same vertex are labeled the same.

Let $\CB$ be a non-empty subset of $2^{E^0}$.  Given a set $A \in \CB$ we write $\CL(AE^1)$ for the set $\{\CL(e): e \in E^1 \text{ and } s(a) \in A\}$.

For a set $A \in \CB$ and a word $\alpha \in \CL^n(E)$ we define the relative range of $\alpha$ with respect to $A$ as $$r(A,\alpha):= \{r(p): \CL(p) = \alpha \text{ and } s(p)\in A\}.$$

We say $\CB$ is closed under relative ranges if $r(A,\alpha) \in \CB$ for any $A\in\CB$ and any $\alpha \in \CL^n(E)$.

$\CB$ is said to be \textbf{\textit{accommodating}} if
\begin{enumerate}
    \item $r(\alpha)\in \CB$ for each $\alpha \in \CL^*(E)$
    \item $\CB$ is closed under relative ranges
    \item $\CB$ is closed under finite intersections and unions.
\end{enumerate}

If $\CB$ is accommodating for $(E, \CL)$, the triple $\ls$ is called a labelled space.  For trivial reasons, we will assume that $\CB \neq \{\emptyset\}$

A labeled space $\ls$ is called \textit{\textbf{weakly left--resolving}} if for any $A,~B \in \CB$ and any $\alpha \in \CL^*(E)$ $$r(A \cap B, \alpha) = r(A,\alpha) \cap r(B,\alpha).$$

We say $\ls$ is \textit{\textbf{non--degenerate}} if $\CB$ is closed under relative complements.  A \textit{\textbf{normal}} labeled space is accommodating and non--degenerate.

\section{Labeled Graph $C^*$-algebras}\label{preliminaries2}

Let $\ls$ be a weakly left--resolving labeled space.  A representation of $\ls$ in a $C^*$-algebra consists of projections $\{p_A : A \in \CB \}$,
and partial isometries $\{ s_a : a \in \CA\}$, satisfying:
\begin{enumerate}
\item For any $A,~B \in \CB$, $p_Ap_B = p_{A\cap B}$, and  $p_{A\cup B} = p_A+p_B-p_{A\cap B}$.

\item For any $a,~b \in \CA$, $s_a^*s_b = p_{r(a)}\delta_{a,b}$.

\item For any $a \in \CA$ and $A \in \CB$, $s_a^*p_A = p_{r(A,a)}s_a^*$.

\item For $A \in \CB$ with $\CL(AE^1)$ finite and $A$ does not contain a sink, we have $$p_A = \sum_{a \in \CL(AE^1)}s_ap_{r(A,a)}s_a^*.$$

\end{enumerate}
The \textit{labeled graph $C^*$-algebra} is the $C^*$-algebra generated by a
universal representation of $\ls$.  For a word $\mu = a_1 \cdots
a_n,$ we write $s_\mu$ to mean $s_{a_1} \cdots s_{a_n}$.  
\begin{remark}
Given $A\in \CB$.  If $a,~b \in \CL(AE^1),$ then $s_ap_{r(A,a)}s_a^*\cdot s_bp_{r(A,b)}s_b^* = \delta_{a,b}s_ap_{r(A,a)}s_a^*$.
    Moreover, $p_As_ap_{r(A,a)}s_a^* = s_ap_{r(A,a)}p_{r(A,a)}s_a^*=s_ap_{r(A,a)}s_a^*$.  Therefore, for any finite subset $\CS$ of $\CL(AE^1)$ we have $$p_A \geq \sum_{a \in \CS}s_ap_{r(A,a)}s_a^*.$$

\end{remark}

The result of the following lemma is similar to results obtained for graph $C^*$-algebras and other related algebras (see \cite[Lemma 1.1]{KPR} or \cite{E2}).

One easily checks from the relations that $s_\mu^* s_\mu = p_{r(\mu)}$ and that $s_\nu^* s_\mu = 0$ unless one of $\mu$, $\nu$ extends the other. In fact,  

\begin{lemma} \label{lemma33.BES}  Let $\mu,~\nu \in \CL^*(E)$.  Then
\begin{displaymath}
\begin{array}{ll}
s_\mu^*s_\nu     &  =  \left\{\begin{array}{ll}
        p_{r(\mu)} & \textrm{if $\nu = \mu$}\\
        p_{r(\mu)}s_\gamma & \textrm{if $\nu = \mu\gamma$}\\
        s_\gamma^*p_{r(\nu)}  & \textrm{if $\mu = \nu\gamma$}\\
        0 & \text{otherwise}
        \end{array} \right. \\
\end{array}
\end{displaymath}
This, in turn, gives us:
\begin{displaymath}
\begin{array}{ll}
s_\mu^*s_\nu     &  =  \left\{\begin{array}{ll}
        p_{r(\mu)} & \textrm{if $\nu = \mu$}\\
        s_\gamma p_{r(\nu)} & \textrm{if $\nu = \mu\gamma$}\\
        p_{r(\mu)}s_\gamma^* & \textrm{if $\mu = \nu\gamma$}\\
        0 & \text{otherwise.}
        \end{array} \right. \\
\end{array}
\end{displaymath}

\end{lemma}
\begin{proof}
If $\mu = \nu,$ then $s_\mu^*s_\nu = s_\mu^*s_\mu = p_{r(\mu)}$.

If $\nu = \mu\gamma$, then $s_\mu^*s_\nu = s_\mu^*s_{\mu\gamma}= s_\mu^*s_\mu s_\gamma = p_{r(\mu)}s_\gamma = s_\gamma p_{r(r(\mu),\gamma)} = s_\gamma p_{r(\mu\gamma)} = s_\gamma p_{r(\nu)}$.

If $\mu = \nu\gamma$, then $s_\mu^*s_\nu = s_{\nu\gamma}^*s_\nu = (s_\nu s_\gamma)^*s_\nu = s_\gamma^*s_\nu^*s_\nu = s_\gamma^*p_{r(\nu)} = p_{r(\nu\gamma)}s_\gamma^* = p_{r(\mu)}s_\gamma^*$.

Lastly, suppose $\mu = a_1a_2\ldots a_m$, $\nu = b_1b_2 \ldots b_n$, and neither $\mu$ nor $\nu$ extends the other.  This means $a_k \neq b_k$ for some $k \leq \min\{m,~n\}$.  We will assume that $k$ is the smallest such index, that is, $a_i = b_i$ for $i \leq k-1$, but $a_k \neq b_k$.  Then 
\begin{align*}
 s_\mu^*s_\nu & = (s_{a_1a_2\ldots a_{k-1}a_k \ldots a_m})^*s_{b_1b_2\ldots b_{k-1}b_k\ldots b_n}\\   
 & = (s_{a_k \ldots a_m})^*(s_{a_1a_2\ldots a_{k-1}})^*s_{b_1b_2\ldots b_{k-1}}s_{b_k\ldots b_n}\\
 & = (s_{a_k \ldots a_m})^*p_{r(b_1b_2\ldots b_{k-1})}s_{b_k\ldots b_n}, \text{ since } a_i = b_i, \text{ for } i <  k\\
 & = (s_{a_k \ldots a_m})^*s_{b_k\ldots b_n}p_{r(b_1b_2\ldots a_n)}\\
 & = s_{a_m}^* \ldots s_{a_k}^*s_{b_k}\ldots s_{b_n}p_{r(b_1b_2\ldots a_n)}\\
 & = 0 \text{ because } a_k \neq b_k.
\end{align*}
\end{proof}

Using $\epsilon$ to denote the empty word, and $\CL^{\#}(E)$ to denote $\CL(E^*) \cup \{\epsilon\}$, we find that
\[
C^*(E,\CL,\CB) = \overline{\text{span}} \{s_\mu p_A s_\nu^* : \mu,\;\nu \in \CL^{\#}(E) \text{ and } A\in\CB \}.
\]

Here we use $s_\epsilon$ to mean the unit element of the multiplier algebra of $C^*(E,\CL,\CB)$.

In \cite{JKP}, they provided the a definition of ``hereditary" subset of $\CB$.  We will restate it here; note that $E$ may contain sinks.

\begin{definition} \label{heredirary} For a subset $\CH$ of $\CB$ we say that $\CH$ is hereditary if it satisfies the following:
\begin{enumerate}
    \item for any $A\in \CH$ and for any $\alpha \in \CL^*(E),$ we have that  $r(A,\alpha) \in \CH$.
    \item $A \cup B \in \CH$ whenever $A,B \in \CH$.
    \item If $A \in \CH$ and $B\in \CB$ with $B \subseteq A,$ then $B \in \CH$.
\end{enumerate}
\end{definition}

Notice that, in addition to being closed under finite unions, $\CH$ is closed under arbitrary intersections.  Moreover, when $\ls$ is normal, if $A\in \CH$ and $B \in \CB,$ then $A \setminus B \in \CH$.

In the discussions to follow, we will assume that $\CH$ does not contain the empty set.  Thus, we modify (1) of Definition \ref{heredirary} to: $r(A, \alpha) \in \CH$, whenever $r(A, \alpha)$ is non--empty.



\section{Main results}

We start with a proposition on sub-algebras of $C^*\ls$. This is an important result on its own, and we will be using it to build the $C^*$-algebras of labeled graphs when the graphs are finite graphs with no loops or cycles. 

\begin{proposition} \label{Morita.eq} Let
$\ls$ be a labeled space and let $\CH$ be a
hereditary subset of $\CB$. Suppose also that $\forall a \in \CA,$ either $s(a)\in \CH$ or $\forall A \in \CH$ $s(a) \cap A = \emptyset$.
Then
$$I := \overline{span}\{s_\alpha p_A s_\beta^*:\alpha,\beta \in \CL^{\#}(E),
A \in \CH\}$$ is an ideal of $C^*\ls$ which is
Morita equivalent to
\begin{align*}
 \CC  :=  \overline{span} \{ & s_\alpha p_A s_\beta^*:  s(\alpha) \in\CH \text{ or } \alpha = \epsilon, \\   
 & s(\beta) \in \CH  \text{ or } \beta = \epsilon,  \text{ and } A\in \CH \}.
\end{align*}

\end{proposition}
\begin{proof}

Let $A\in \CH,~B \in \CB$, and let $\alpha,~\beta,~\mu,~\nu \in \CL^{\#}(E)$.  Then $s_\alpha p_A s_\beta^* s_\mu p_B s_\nu^*$ is either zero or is of the form $s_{\alpha'} p_C s_{\nu'},$ where $C \subseteq A$ or $C \subseteq r(A,\gamma)$, for some $\gamma \in \CL^*(E)$; either way, $C \in \CH$. Moreover $\alpha'$ extends $\alpha$ and $\nu'$ extends $\nu$ (or they are equal). Thus $I$ is an ideal.  
Using a similar argument we see that  
$$\CX  :=  \overline{span} \{s_\alpha p_A s_\beta^*:  s(\alpha) \in \CH \text{ or } \alpha = \epsilon, \beta \in \CL^{\#}(E)\}$$ is a right ideal of $C^*\ls$. 

To show that $\CC$ is an algebra, let $s_\alpha p_A s_\beta^*, s_\mu p_B s_\nu^* \in \CC$ then the product $(s_\alpha p_A s_\beta^*) (s_\mu p_B s_\nu^*)$ is zero unless $\beta$ or $\mu$ extends the other.  In that case, say $\beta = \mu\gamma$ then $s_\alpha p_A s_\beta^* s_\mu p_B s_\nu^*=s_\alpha p_A s_{\mu\gamma}^* s_\mu p_B s_\nu^* = s_\alpha p_A s_\gamma^* p_{r(\mu)} p_B s_\nu^* = s_\alpha p_A p_{r(\beta)} p_{r(B,\gamma)}s_\gamma^*s_\nu^* = s_\alpha p_A p_{r(\beta)} p_{r(B,\gamma)}s_{\nu\gamma}^* = s_\alpha p_{A \cap r(\beta) \cap r(B,\gamma)}s_{\nu\gamma}^*$.  Now, either $\nu = \epsilon$ or  $s(\nu\gamma) \subseteq s(\nu) \in \CH \Rightarrow s(\nu\gamma) \in \CH$.  In the case $\nu = \epsilon$, looking at $p_{r(B,\gamma)}$, either the product is zero or $B \cap s(\gamma) \neq \emptyset$ which implies that $s(\gamma) \in \CH$.  Therefore $(s_\alpha p_A s_\beta^*)(s_\mu p_B s_\nu^*) \in \CC$.

To show that  $\CX\CX^* = \CC$, take $s_\alpha p_A s_\beta^*, s_\mu p_B s_\nu^* \in \CX$.  Then $s_\alpha p_A s_\beta^*\cdot (s_\mu p_B s_\nu^*)^* = s_\alpha p_A s_\beta^*\cdot s_\nu p_B s_\mu^*$.  This is zero unless $\beta$ or $\nu$ extends the other.  In that case, say $\beta = \nu\gamma$, then  $s_\alpha p_A s_\beta^*\cdot (s_\mu p_B s_\nu^*)^* = s_\alpha p_A  p_{r(\nu)} s_\gamma^* p_B s_\mu^* = s_\alpha p_A p_{r(\nu)} p_{r(B,\gamma)} s_\gamma^*  s_\mu^* = s_\alpha p_{A \cap r(\nu) \cap r(B,\gamma)} s_{\mu \gamma}^*$.  But $s(\mu \gamma) \subseteq s(\mu) \in \CH \Rightarrow s(\mu \gamma) \in \CH$, thus $s_\alpha p_{A \cap r(\nu) \cap r(B,\gamma)} s_{\mu \gamma}^* \in \CC$.   So $\CX\CX^* \subseteq \CC$.  
If $s_\alpha p_A s_\beta^* \in \CC$, then $s_\alpha p_A s_\beta^* = (s_\alpha p_A s_\epsilon^*)(s_\beta p_A s_\epsilon^*)^* \in \CX\CX^*$.  That is, $\CC \subseteq \CX\CX^*$. Therefore, $\CX\CX^* = \CC$.  Similarly, $\CX^*\CX = I$
\end{proof}
\begin{definition}
Let $\ls$ be any labeled space.  We call a non empty element $A$ of $\CB$ a minimal sinks set if $A \subseteq E^0_{sink}$ and for any $B \in \CB$ either $A \subseteq B$ or $A \cap B = \emptyset$.
\end{definition}

\begin{corollary} \label{Morita.corollary}
For a labeled space $\ls$, let $A\in\CB$ be a minimal sinks set.   Then, $$I_A = \overline{span}\{s_\alpha p_As_\beta^* : \alpha, \beta \in \CL^{\#}(E)\}$$ is a closed two sided ideal of $C^*\ls,$ which is isomorphic to $\CK(\ell^2(\CL^*(A)))$, where $\CL^*(A) := \{\alpha: \alpha \in \CL^*(E) \text{ and } r(\alpha) \supseteq A, \text{ or }  \alpha = \epsilon\}$.
\end{corollary}
\begin{proof} Notice that the singleton set $\CH = \{A\}$ is hereditary, thus $I_A$ is a two sided ideal. If $\alpha, \beta, \mu, \nu \in \CL^{\#}(E)$, using Lemma \ref{lemma33.BES}, 

\begin{displaymath}
\begin{array}{ll}
s_\alpha p_As_\beta^*\cdot s_\mu p_As_\nu^*     &  =  \left\{\begin{array}{ll}
        s_\alpha p_A p_{r(\beta)} p_A s_\nu^* & \textrm{if $\beta = \mu$}\\
        s_\alpha p_A s_\gamma p_{r(\beta)}  p_A s_\nu^* & \textrm{if $\beta = \mu\gamma$}\\
        s_\alpha p_A p_{r(\mu)}s_\gamma^*  p_A s_\nu^* & \textrm{if $\mu = \beta\gamma$}\\
        0 & \text{otherwise.}
        \end{array} \right. \\
\end{array}
\end{displaymath}
However, $s_\alpha p_A p_{r(\beta)} p_A s_\nu^* = s_\alpha p_{A \cap r(\beta)} s_\nu^* = s_\alpha p_A s_\nu^*$, if $r(\alpha), r(\beta), r(\nu) \supseteq A$; otherwise  $s_\alpha p_A p_{r(\beta)} p_A s_\nu^* = 0$.\\ $s_\alpha p_A s_\gamma p_{r(\beta)}  p_A s_\nu^* = s_\alpha  s_\gamma p_{r(A, \gamma)} p_{r(\beta)}  p_A s_\nu^* = 0$, since $A \subseteq E^0_{sink}$.  Similarly, $s_\alpha p_A p_{r(\mu)}s_\gamma^*  p_A s_\nu^* = s_\alpha p_A p_{r(\mu)} p_{r(A, \gamma)} s_\gamma^* s_\nu^* = 0$.\\ Thus, $\{s_\alpha p_As_\beta^*:r(\alpha), r(\beta) \supseteq A,~\alpha = \epsilon, \text{ or } \beta = \epsilon\}$ forms a set of matrix units.  Therefore, $I_A \cong \CK(\ell^2(\CL^*(A)))$.
\end{proof}

We take a short pause to look at a couple of examples; one labeled graph with two labeled spaces.
\begin{eg}
Consider the labeled graph: $$
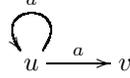
\xymatrix{{u} \ar@(ur,ul)_{a}[] \ar [r]^a & {v}}$$ Let $\CB_1 = \{ \{u, v\}, \emptyset \}$. $(E, \CL, \CB_1)$ is a normal labeled space.  Recall that $p_A s_\alpha = s_\alpha p_{r(A,\alpha)}$ and $s_\alpha^* p_A = p_{r(A,\alpha)}s_\alpha^*$.  In this case, for $A = \{u, v\}$ and any word $\alpha =aa\ldots a$, since $r(A,\alpha) = A$, we get $p_A s_\alpha = s_\alpha p_A$ and $s_\alpha^* p_A = p_As_\alpha^*$.  Therefore, if $n > m$, $s_a^n p_A (s_a^*)^m = s_a^n (s_a^*)^m p_A = s_a^{n-m} p_{r(a^m)}p_A = s_a^{n-m}p_A$.  If $n = m$, $s_a^n p_A (s_a^*)^m = s_a^n (s_a^*)^m p_A = p_A$.  And if $n < m$, $s_a^n p_A (s_a^*)^m = p_A s_a^n (s_a^*)^m = p_A p_{r(a^n)}(s_a^*)^{m-n} = p_A  (s_a^*)^{m-n}$. Therefore, $C^*(E,\CL,\CB_1) = \overline{span}\{s_a^k p_A,~ p_A,~ p_A(s_a^*)^k:k\in \IN\} \cong C(\IT).$ In this example, $E^0_{sink} \notin \CB_1$.

Now, let $\CB_2 = \{ \{u\}, \{v\}, \{u, v\}, \emptyset \}$. $(E, \CL, \CB_2)$ is also a normal labeled space. For $A = \{v\}$, $I_A = \overline{span}\{s_\alpha p_As_\beta^* : \alpha, \beta \in \CL^{\#}(E)\} = \overline{span}\{s_a^n p_A(s_a^*)^m : n, m \in \IN \cup \{0\}\} \cong \CK(\ell^2(\IN))$ is an ideal of $C^*(E, \CL, \CB_2)$.  Since $p_{\{u,v\}} = p_{\{u\}} + p_{\{v\}}$, the $C^*$-algebra is generated by $\{s_a^np_{\{u\}}(s_a^*)^m, s_a^np_{\{v\}}(s_a^*)^m : n.m \in \IN \cup \{0\}\}$.  

For $n,m,k,l \in \IN \cup \{0\},$ we will compute $s_a^np_{\{u\}}(s_a^*)^m\cdot s_a^kp_{\{u\}}(s_a^*)^l$. If $m = k$, 
\begin{align*}
 s_a^np_{\{u\}}(s_a^*)^m\cdot s_a^kp_{\{u\}}(s_a^*)^l & = s_a^np_{\{u\}}p_{r(a^m)}p_{\{u\}}(s_a^*)^l\\   
 & = s_a^np_{\{u\}}p_{\{u,v\}}p_{\{u\}}(s_a^*)^l\\
 & = s_a^np_{\{u\}}(s_a^*)^l.
\end{align*}
If $m > k$, 
\begin{align*}
 s_a^np_{\{u\}}(s_a^*)^m\cdot s_a^kp_{\{u\}}(s_a^*)^l & = s_a^np_{\{u\}}p_{\{u,v\}}p_{r(\{u\},a^{m-k})}(s_a^*)^{m-k}(s_a^*)^l\\
 & = s_a^np_{\{u\}}(s_a^*)^{l+m-k}.
\end{align*}
Similarly, (taking the adjoint), if $m < k$,

$s_a^np_{\{u\}}(s_a^*)^m\cdot s_a^kp_{\{v\}}(s_a^*)^l = s_a^{n+k-m}p_{\{u\}}(s_a^*)^l$.\\
Thus, $\{s_a^np_{\{u\}}(s_a^*)^m\}$ is a set of matrix units. Therefore, $\overline{span}\{s_a^np_{\{u\}}(s_a^*)^m : n, m \in \IN \cup \{0\}\} \cong \CK(\ell^2(\IN))$.

Next we will compute $s_a^np_{\{u\}}(s_a^*)^m\cdot s_a^kp_{\{v\}}(s_a^*)^l$.\\
If $m = k$, 

$s_a^np_{\{u\}}(s_a^*)^m\cdot s_a^kp_{\{v\}}(s_a^*)^l = s_a^np_{\{u\}}p_{\{u,v\}}p_{\{v\}}(s_a^*)^l = 0$.\\
If $m > k$, 
\begin{align*}
 s_a^np_{\{u\}}(s_a^*)^m\cdot s_a^kp_{\{v\}}(s_a^*)^l & = s_a^np_{\{u\}}p_{\{u,v\}}p_{r(\{v\},a^{m-k})}(s_a^*)^{m-k}(s_a^*)^l\\ 
 & = 0,\text{ since } p_{r(\{v\},a^{m-k})} = 0.
\end{align*}
Similarly, if $m < k$, $s_a^np_{\{u\}}(s_a^*)^m\cdot s_a^kp_{\{v\}}(s_a^*)^l = 0$.\\
This gives us 
\begin{align*}
C^*(E, \CL, \CB_2) & \cong \overline{span}\{s_a^np_{\{u\}}(s_a^*)^m\} \oplus \overline{span}\{s_a^np_{\{v\}}(s_a^*)^m\}\\ 
 & \cong \CK(\ell^2(\IN)) \oplus \CK(\ell^2(\IN)).
\end{align*}

\end{eg}

Next we will write out the $C^*$-algebra of a finite labeled graph when the graph has no loops or cycles.  This characterization is similar to \cite[Corollary 2.3]{KPR} for a finite directed graph with no loops.
In the proposition, we will assume that $E^0_{sink} \in \CB$ and that $\ls$ is normal. Afterwards, with the help of examples, we will see the significance of these assumptions.

\begin{proposition} \label{direct.sum}
Let $E$ be a finite graph with no loops (or cycles).  Suppose also that $E^0_{sink} \in \CB$, where $\ls$ is a normal labeled space. For each $x \in E^0_{sink}$ compute $A_x = \cap_{x\in A}A$.  Let $\CV=\{A_x:x\in E^0_{sink}\}$ . Then $$C^*\ls \cong \bigoplus_{V \in \CV}\CM_{n_V+1}(\IC),$$ where $n_V = $ the number of words $\alpha$ with $V \subseteq r(\alpha)$.
\end{proposition}

\begin{proof}
When $x \in E^0_{sink}$ then $A_x = \cap_{x\in A}A$ is a finite intersection of elements of $\CB$, hence is in $\CB$, moreover $A_x \subseteq E^0_{sink}$.  The collection $\CV=\{A_x:x\in E^0_{sink}\}$ is a finite collection of mutually disjoint members of $\CB$.  Also, if $V \in \CV$ and $B \in \CB$ either $V \subseteq B$ or $V \cap B = \emptyset$.  Then by Corollary \ref{Morita.corollary}, each $V \ \in \CV$ gives us an ideal $I_V$ of $C^*\ls$. 

Now, let $\alpha,~ \beta \in \CL^{\#}(E)$ and $A \in \CB$.  Write $A = A_1 \cup C_1$, where $C_1 = A\cap E^0_{sink}$ and $A_1 = A \setminus C_1$. \\
Then $s_\alpha p_A s_\beta^* =s_\alpha p_{A_1} s_\beta^* +s_\alpha p_{C_1} s_\beta^*$.  Notice that $C_1$ is a disjoint union of elements of $\CV$, thus $s_\alpha p_{C_1} s_\beta^* = \sum_{V \in \CV, V \subseteq C_1}s_\alpha p_V s_\beta^*$.\\ \\ \\ Moreover $s_\alpha p_{A_1} s_\beta^* = \sum_{a \in \CL(A_1E^1)}s_{\alpha a} p_{r(A_1,a)}s_{\beta a}^*= \sum_{a \in \CL(AE^1)}s_{\alpha a} p_{r(A,a)}s_{\beta a}^*$.\\ \\
For each $a\in \CL^*(AE^1)$ write $r(A,a) = A_2 \cup C_2$, where $C_2 = r(A,a)\cap E^0_{sink}$ and $A_2 = r(A,a) \setminus C_2$.  Then $s_{\alpha a} p_{r(A,a)} s_{\beta a} = s_{\alpha a} p_{A_2} s_{\beta a} +s_{\alpha a} p_{C_2} s_{\beta }^*$. This implies $s_{\alpha a} p_{A_2} s_{\beta a} = \sum_{ab \in \CL(AE^2)}s_{\alpha ab} p_{r(A,ab)}s_{\beta ab}^*$, and $s_{\alpha a} p_{C_2} s_{\beta a}^* = \sum_{V \in \CV, V \subseteq C_2}s_{\alpha a} p_V s_{\beta a}^*$. \\ Since $E$ is a finite graph with no loops, repeating this process will eventually give us $s_\alpha p_A s_\beta^*$ as a sum of elements of the form $s_{\alpha \gamma} p_V s_{\beta \gamma}^*$ where $V \in \CV$.  For each $V \in \CV$ there are $n_V$ distinct words $\alpha$ in $\CL^*(E)$ with $r(\alpha) \supseteq V$, adding 1 for the empty word $\epsilon$, we get that $I_V \cong \CM_{n_v+1}(\IC)$. For $U, V \in \CV$, $s_\alpha p_U s_\beta^* \cdot s_\mu p_V s_\nu^* = 0$ unless $U = V$, hence the direct sum.
\end{proof}

In the next, rather trivial, example we, once again, see the significance of having $E^0_{sink}$ in $\CB$.

\begin{eg}  Consider the labeled graph:
$$\xymatrix{ & v  \\ u \ar[r]_{a} \ar[ur]^{a} & w \ar[r]_{b} & x}$$ Notice that $r(a) = \{v,w\},~ r(ab) = \{x\}$, and $E^0_{sink} = \{v,x\}$.  The smallest normal labeled space $(E, \CL, \CB_1)$ containing these sets has $\CB_1 = \{\{v,w\}, \{x\}, \{v, w, x\}, \{w\}, \{v\}, \emptyset\}$.  Using  Proposition \ref{direct.sum}, we get $C^*(E, \CL, \CB_1) = \CM_2(\IC) \oplus \CM_3(\IC)$.  On the other hand, if $\CB_2 \cong \{\{x\}, \{v,w\},\{v,w,x\}, \emptyset\}$, which still makes a normal labeled space, it is not difficult to see that $C^*(E, \CL, \CB_2) \cong \CM_3(\IC)$.
\end{eg}

\begin{eg}  Consider the labeled graph:
$$
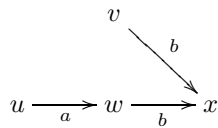
\xymatrix{ & v  \ar[dr]^{b}  \\ u \ar[r]_{a} & w \ar[r]_{b} & x}$$ Notice that $r(a) = \{w\},~ r(ab) = \{x\}$, and $E^0_{sink} = \{x\}$. Consider $\CB = \{\{v,w\}, \{w\}, \{x\}, \emptyset\}$,  $\ls$ is a labeled space (not normal).  It is not difficult to see that $C^*(E, \CL, \CB) = \CM_3(\IC)$. If we attempt to make this space normal by adding $\{v\}$ into $\CB$, the resulting set would not make $\ls$ weakly left--resolving.
\end{eg}

\begin{remark}
The next natural question would be ``can one normalize a labeled space?"  More precisely, given a labeled space $(E,\CL,\CB_1)$ is there a normal labeled space $(E,\CL, \CB_2)$ such that $C^*(E,\CL,\CB_1) \cong C^*(E,\CL,\CB_2)$? 
\end{remark}



\begin{thebibliography}{99}         


\bibitem{BP1} {T. Bates and D. Pask,} {\em $C^*$-algebras of labeled graphs}, J. Operator Theory. \textbf{57}(2007), 101--120.

\bibitem{BP2} {T. Bates and D. Pask,} {\em $C^*$-algebras of labeled graphs II - simplicity results},  Math. Scand. \textbf{104}(2009), no. 2, 249--274.

\bibitem {BPRS}  {\scshape Bates, T.; Pask, D.; Raeburn, I.; Szyma\'nski, W.} The $C^*$-algebras of row--finite Graphs. {\em New York J. Math.} {\bf 6} (2000), 307--324.

\bibitem {BEF}  {\scshape Bhat, R.; Elliott G.; Fillmore, P.}
Lectures on Operator Theory. {\em American Mathematical Society,} (2000).



\bibitem {Cu}  {\scshape Cuntz, J.} Simple $C^*$-algebras Generated by
Isometries {\em Commun. Math. Phys.} {bf 57} (1977), 173--185.

\bibitem{CK} {\scshape Cuntz, J.; Krieger, W.} A class of $C^*$-algebras and
topological Markov  chains, {\em Invent. Math.} {\bf 56} (1980),
251--268.

\bibitem {D}  {\scshape Dixmier, J.} $C^*$-algebras. {\em North-Holland Publishing Co.,} 1977.

\bibitem {DT}  {\scshape Drinen, D.; Tomforde, M.} The $C^*$-algebras of Arbitrary Graphs. {\em Rocky Mountain J. Math.} {\bf 35} (2005), no. 1, 105--135. 


\bibitem {E2} {\scshape Ephrem, M.} $C^*$-Algebra of The
$\mathbb{Z}^n$ Tree, {New York J. Math.} {\bf 17}  (2011) 1 -- 20

\bibitem {E3} {\scshape Ephrem, M.} Primitive Ideals of Labeled Graph $C^*$-algebras, {\em Houston Journal of Math. to appear}.

\bibitem {JKK} {\scshape Jeong, J. A.; Kang, E. J.; Kim, S. H.}
AF labeled graph $C^*$-algebras, {\em J. Funct. Anal.} {\bf 266}, (2014), 2153--2173.

\bibitem {JKP} {\scshape Jeong, J. A.; Kim, S. H.; Park, G. H.} The structure of gauge-invariant ideals of labeled graph $C^*$-algebras, {\em J. Funct. Anal.} {\bf 262}, (2012), no. 4, 1759--1780. 

\bibitem {JP} {\scshape Jeong, J.; Park, G. H.}
Simple labeled graph $C^*$-algebras are associated to disagreeable labeled spaces, {\em J. Mathematical Anal. and Appl.} {\bf 461}, (2017), no. 2, 1391--1403.


\bibitem {KPR} {\scshape Kumjian, A.; Pask, D.; Raeburn, I.}
Cuntz-Krieger Algebras of Directed Graphs, {\em Pacific
J. Math.} {\bf 184}, (1998), 161--174.

\bibitem {KPRR} {\scshape Kumjian, A.; Pask, D.; Raeburn, I. ; Renault, J.}
Graphs, Groupoids and Cuntz-Krieger Algebras,
{\em J. Funct. Anal.} {\bf 144}, (1997), 505--541.

\bibitem {M}  {\scshape Murphy, G.} $C^*$-algebras And Operator
Theory.  {\em Academic Press,} 1990.









\bibitem {T} {\scshape Tomforde, M.} A unified approach to Exel-Laca algebras and $C^*$-algebras associated to graphs, {\em J. Operator Theory.} {\bf 50} (2003) No. 2, 345--368.



\end{thebibliography}
\end{document}